\documentclass[11pt,reqno]{amsart}

\usepackage{graphicx}
\usepackage[top=20mm, bottom=20mm, left=30mm, right=30mm]{geometry}
\usepackage[colorlinks=true,citecolor=black,linkcolor=black,urlcolor=magenta]{hyperref}

\usepackage{fancyvrb}
\usepackage{ulem} 
\usepackage[T1]{fontenc} 
\usepackage{pbsi} 
\usepackage{url}

\newcommand{\citep}{\cite} 
\newcommand{\cf}[0]{\textit{cf}.\ } 
\newcommand{\Section}[0]{\S} 
\renewcommand{\emph}[1]{\textit{#1}}

\theoremstyle{plain}
\newtheorem{theorem}{Theorem}[section]

\newtheorem{cor}[theorem]{Corollary}
\newtheorem{prop}[theorem]{Proposition}

\theoremstyle{definition}
\newtheorem{definition}[theorem]{Definition}
\newtheorem{example}[theorem]{Example}

\theoremstyle{remark}
\newtheorem{remark}[theorem]{Remark}

\usepackage{enumitem}
\setlist[itemize]{leftmargin=0.35in}

\newcommand{\gkpSI}[2]{\ensuremath{\genfrac{\lbrack}{\rbrack}{0pt}{}{#1}{#2}}} 
\newcommand{\gkpSII}[2]{\ensuremath{\genfrac{\lbrace}{\rbrace}{0pt}{}{#1}{#2}}} 
\newcommand{\gkpEII}[2]{\ensuremath{\left\langle\genfrac{\langle}{\rangle}{
            0pt}{}{#1}{#2}\right\rangle}} 
\newcommand{\FcfII}[3]{\ensuremath{\gkpSI{#2}{#3}_{#1}}} 
\newcommand{\Iverson}[1]{\ensuremath{\left[#1\right]_{\delta}}} 

\DeclareMathOperator{\numpoly}{num} 
\DeclareMathOperator{\denompoly}{denom} 

\title[Combinatorial Identities for Generalized Stirling Numbers]{
       Combinatorial Identities for 
       Generalized Stirling Numbers Expanding 
       $f$-Factorial Functions and the $f$-Harmonic Numbers} 

\author[Maxie D. Schmidt]{Maxie D. Schmidt \\ \\ 
        School of Mathematics \\ 
        Georgia Institute of Technology \\ 
        117 Skiles Building \\ 
        686 Cherry Street NW \\ 
        Atlanta, GA 30332 \\ \\ 
        \texttt{maxieds@gmail.com}}         
\address{School of Mathematics, Georgia Institute of Technology, Atlanta, GA 30332}
\email{maxieds@gmail.com} 
\thanks{} 
       
\date{2017.03.29-v1}



\allowdisplaybreaks

\begin{document}

\begin{abstract}
We introduce a class of $f(t)$-factorials, or $f(t)$-Pochhammer symbols, that 
includes many, if not most, well-known factorial and multiple factorial 
function variants as special cases. We consider the combinatorial properties 
of the corresponding generalized classes of 
Stirling numbers of the first kind which arise as the coefficients of the 
symbolic polynomial expansions of these $f$-factorial functions. 
The combinatorial properties of these more general 
parameterized Stirling number triangles we prove within the article 
include analogs to known expansions of the ordinary Stirling numbers by 
$p$-order harmonic number sequences through the definition of a 
corresponding class of $p$-order $f$-harmonic numbers. 
\end{abstract} 

\maketitle

\section{Introduction} 

\subsection{Generalized $f$-factorial functions} 

\subsubsection*{Definitions} 

For any function, $f: \mathbb{N} \rightarrow \mathbb{C}$, and fixed 
non-zero indeterminates $x, t \in \mathbb{C}$, we introduce and define the 
\emph{generalized $f(t)$-factorial function}, or alternately the 
\emph{$f(t)$-Pochhammer symbol}, denoted by $(x)_{f(t),n}$, 
as the following products: 
\begin{align} 
\label{eqn_xft_gen_PHSymbol_def} 
(x)_{f(t),n} & = 
     \prod_{k=1}^{n-1} \left(x + \frac{f(k)}{t^k}\right). 
\end{align} 
Within this article, 
we are interested in the combinatorial properties of the coefficients of the 
powers of $x$ in the last product expansions which we consider to be 
generalized forms of the \emph{Stirling numbers of the first kind} in 
this setting. 
Section \ref{subSection_Intro_GenSNumsDefs} defines generalized 
Stirling numbers of both the first and second kinds and motivates the 
definitions of auxiliary triangles by special classes of formal power series 
generating function transformations and their corresponding 
negative-order variants considered in the references 
\citep{GFTRANS2016,GFTRANSHZETA2016}. 

\subsubsection*{Special cases} 

Key to the formulation of applications and interpreting the generalized 
results in this article is the observation that the 
definition of \eqref{eqn_xft_gen_PHSymbol_def} provides an effective 
generalization of almost all other related factorial function variants 
considered in the references when $t \equiv 1$. 
The special cases of $f(n) := \alpha n+\beta$ for some 
integer-valued $\alpha \geq 1$ and $0 \leq \beta < \alpha$ lead to the 
motivations for studying these more general factorial functions in 
\citep{GFTRANSHZETA2016}, and form the expansions of multiple 
$\alpha$-factorial functions, $n!_{(\alpha)}$, studied in the 
triangular coefficient expansions defined by 
\citep{MULTIFACT-CFRACS,MULTIFACTJIS}. 
The \emph{factorial powers}, or 
\emph{generalized factorials of $t$ of order $n$ and increment $h$}, 
denoted by $t^{(n, h)}$, or the \emph{Pochhammer k-symbol} denoted by 
$(x)_{n,h} \equiv p_n(h, t) = t(t+h)(t+2h)\cdots(t+(n-1)h)$, studied in 
\citep{q-DIFFSq-FACTS,MULTIFACT-CFRACS,CK} form particular special cases, 
as do the the forms of the generalized \emph{Roman factorials} and 
\emph{Knuth factorials} for $n \geq 1$ defined in \citep{LOEBBINOM}, and the 
\emph{$q$-shifted factorial functions} considered in 
\citep{q-SHIFTEDFACTS,q-DIFFSq-FACTS}. 
When $(f(n), t) \equiv (q^{n+1}, 1)$ these products are related to the 
expansions of the finite cases of the \emph{$q$-Pochhammer symbol} products, 
$(a; q)_n = (1-a)(1-aq)\cdots(1-aq^{n-1})$, and the corresponding 
definitions of the generalized Stirling number triangles defined in 
\eqref{eqn_genS1ft_rec_def} of the next subsection are precisely the 
\emph{Gaussian polynomials}, or \emph{$q$-binomial coefficients}, 
studied in relation to the $q$-series expansions and 
$q$-hypergeometric functions in \citep[\S 17]{NISTHB}. 

\subsubsection*{New results proved in the article} 

The results proved within this article, for example, provide new 
expansions of these special factorial functions in terms of their 
corresponding \emph{$p$-order $f$-harmonic number sequences}, 
\[
F_n^{(p)}(t) := \sum_{k \leq n} \frac{t^k}{f(k)^p}, 
\]
which generalize known expansions of Stirling numbers by the ordinary 
\emph{$p$-order harmonic numbers}, 
$H_n^{(p)} \equiv \sum_{1 \leq k \leq n} k^{-r}$, in 
\citep{STIRESUMS,MULTIFACTJIS,GFTRANS2016,GFTRANSHZETA2016}. 
Still other combinatorial sums and properties satisfied by the 
symbolic polynomial expansions of these special case factorial functions 
follow as corollaries of the new results we prove in the next sections. 
The next subsection precisely expands the generalized factorial expansions 
of \eqref{eqn_xft_gen_PHSymbol_def} through the generalized class of 
Stirling numbers of the first kind defined recursively by 
\eqref{eqn_genS1ft_rec_def} below. 

\subsection{Definitions of generalized $f$-factorial Stirling numbers} 
\label{subSection_Intro_GenSNumsDefs} 

We first employ the next recurrence relation to define the 
generalized triangle of Stirling numbers of the first kind, 
which we denote by $\gkpSI{n}{k}_{f(t)} := [x^{k-1}] (x)_{f(t),n}$, 
or just by $\gkpSI{n}{k}_f$ when the 
context is clear, for natural numbers $n, k \geq 0$
\citep[\cf \S 3.1]{MULTIFACTJIS} \footnote{ 
     The bracket symbol $\Iverson{\mathtt{cond}}$ denotes 
     \emph{Iverson's convention} which evaluates to exactly one of the 
     values in $\{0, 1\}$ and where $\Iverson{\mathtt{cond}} = 1$ 
     if and only if the condition $\mathtt{cond}$ is true. 
}. 
\begin{align} 
\label{eqn_genS1ft_rec_def} 
\gkpSI{n}{k}_{f(t)} & = f(n-1) \cdot t^{1-n} \gkpSI{n-1}{k}_{f(t)} + 
     \gkpSI{n-1}{k-1}_{f(t)} + \Iverson{n = k = 0} 
\end{align} 
We also define the corresponding generalized forms of the 
\emph{Stirling numbers of the second kind}, denoted by 
$\gkpSII{n}{k}_{f(t)}$, so that we can consider inversion relations and 
combinatorial analogs to known identities for the ordinary triangles by the sum 
\begin{align*} 
\gkpSII{n}{k}_{f(t)} & = 
     \sum_{j=0}^{k} \binom{k}{j} \frac{(-1)^{k-j} f(j)^n}{t^{jn} \cdot j!}, 
\end{align*} 
from which we can prove the following form of a particularly useful 
generating function transformation motivated in the references 
when $f(n)$ has a Taylor series expansion in integral powers of $n$ about zero 
\citep[\cf \S 3.3]{MULTIFACTJIS} \citep[\cf \S 7.4]{GKP} 
\citep{SQSERIESMDS,GFTRANSHZETA2016}: 
\begin{align} 
\label{eqn_S2ft_GFTrans_geom_series_exp} 
\sum_{0 \leq j \leq n} \frac{f(j)^k}{t^{jk}} z^j & = 
     \sum_{0 \leq j \leq k} \gkpSII{k}{j}_{f(t)} z^j \times 
     D_z^{(j)}\left[\frac{1-z^{n+1}}{1-z}\right]. 
\end{align} 
The negative-order cases of the infinite series transformation in 
\eqref{eqn_S2ft_GFTrans_geom_series_exp} are motivated in 
\citep{GFTRANSHZETA2016} where we define modified forms of the 
Stirling numbers of the second kind by 
\begin{align*} 
\gkpSII{k}{j}_{f^{\ast}} & = 
     \sum_{1 \leq m \leq j} 
     \binom{j}{m} \frac{(-1)^{j-m}}{j! \cdot f(m)^k}, 
\end{align*} 
which then implies that the transformed ordinary and exponential 
zeta-like power series enumerating 
generalized polylogarithm functions and the 
$f$-harmonic numbers, $F_n^{(p)}(t)$, 
are expanded by the following two series variants \citep{GFTRANSHZETA2016}: 
\begin{align*} 
\sum_{n \geq 1} \frac{z^n}{f(n)^k} & = \sum_{j \geq 0} 
     \gkpSII{k}{j}_{f^{\ast}} \frac{z^j \cdot j!}{(1-z)^{j+1}} \\ 
\sum_{n \geq 1} \frac{F_n^{(r)}(1) z^n}{n!} & = 
     \sum_{j \geq 0} 
     \gkpSII{k}{j}_{f^{\ast}} \frac{z^j \cdot e^{z} (j+1+z)}{(j+1)}. 
\end{align*} 
We focus on the combinatorial relations and sums involving the 
generalized positive-order Stirling numbers in the next few sections. 

\section{Generating functions and expansions by $f$-harmonic numbers} 

\subsection{Motivation from a technique of Euler} 

We are motivated by Euler's original technique for solving the 
\emph{Basel problem} of summing the series, 
$\zeta(2) = \sum_n n^{-2}$, and later more generally all 
even-indexed integer zeta constants, $\zeta(2k)$, in closed-form by 
considering partial products of the sine function \citep[pp. 38-42]{GAMMA}. 
In particular, we observe that we have both an infinite product and a 
corresponding Taylor series expansion in $z$ for $\sin(z)$ given by 
\begin{align*} 
\sin(z) & = \sum_{n \geq 0} \frac{(-1)^n z^{2n+1}}{(2n+1)!} = 
     z \prod_{j \geq 1} \left(1 - \frac{z^2}{j^2 \pi^2}\right). 
\end{align*} 
Then if we combine the form of the coefficients of $z^3$ in the 
partial product expansions at each finite $n \in \mathbb{Z}^{+}$ with the 
known trigonometric series terms defined such that 
$[z^3] \sin(z) = -\frac{1}{3!}$ given on each respective side of the last 
equation, we see inductively that 
\begin{align*} 
H_n^{(2)} = -\pi^2 \cdot [z^2] \prod_{1 \leq j \leq n} 
     \left(1 - \frac{z^2}{j^2 \pi^2}\right) 
     \qquad\longrightarrow\qquad 
     \zeta(2) = \frac{\pi^2}{6}. 
\end{align*} 
In our case, we wish to similarly enumerate the $p$-order $f$-harmonic numbers, 
$F_n^{(p)}(t)$, through the generalized product expansions defined in 
\eqref{eqn_xft_gen_PHSymbol_def}. 

\subsection{Generating the integer order $f$-harmonic numbers} 

We first define a shorthand notation for another form of 
generalized ``\emph{$f$-factorials}'' that we will need in 
expanding the next products as follows: 
\begin{equation*} 
n!_f := \prod_{j=1}^n f(j) \qquad \text{ and } \qquad 
n!_{f(t)} := \prod_{j=1}^{n} \frac{f(j)}{t^j} = 
     \frac{n!_f}{t^{n(n+1)/2}}. 
\end{equation*} 
If we let $\zeta_p \equiv \exp(2\pi\imath / p)$ denote the 
\emph{primitive $p^{th}$ root of unity} for integers $p \geq 1$, and define the 
coefficient generating function, 
$\widetilde{f}_n(w) \equiv \widetilde{f}_n(t; w)$, by 
\begin{align*} 
\widetilde{f}_n(w) & := \sum_{k \geq 2} \gkpSI{n+1}{k}_{f(t)} w^k = 
     \left(\prod_{j=1}^{n} \left(w+f(j) t^{-j}\right) - 
     \gkpSI{n+1}{1}_{f(t)}\right) w, 
\end{align*} 
we can factor the partial products in \eqref{eqn_xft_gen_PHSymbol_def} to 
generate the $p$-order $f$-harmonic numbers in the following forms: 
\begin{align} 
\label{eqn_fkp_partialsum_fCf2_exp_forms} 
\sum_{k=1}^{n} \frac{t^{kp}}{f(k)^p} & = 
     \frac{t^{pn(n+1) / 2}}{\left(n!_{f}\right)^p} 
     [w^{2p}]\left((-1)^{p+1} \prod_{m=0}^{p-1} \sum_{k=0}^{n+1} 
     \FcfII{f(t)}{n+1}{k} \zeta_p^{m(k-1)} w^k\right) \\ 
\notag 
   & = 
     \frac{t^{pn(n+1) / 2}}{\left(n!_{f}\right)^p} 
     [w^{2p}]\left(\sum_{j=0}^{p-1} 
     \frac{(-1)^{j} w^{j}\ p}{p-j} \FcfII{f(t)}{n+1}{1}^j 
     \widetilde{f}_n(w)^{p-j}\right) \\ 
\label{eqn_fkp_partialsum_fCf2_exp_forms_v2} 
\sum_{k=1}^{n} \frac{t^{k}}{f(k)^p} & = 
     \frac{t^{n(n+1) / 2}}{\left(n!_{f}\right)^p} 
     [w^{2p}]\left((-1)^{p+1} \prod_{m=0}^{p-1} \sum_{k=0}^{n+1} 
     \FcfII{f\left(t^{1 / p}\right)}{n+1}{k} \zeta_p^{m(k-1)} w^k\right). 
\end{align} 

\begin{example}[Special Cases] 
For a fixed $f$ and any indeterminate $t \neq 0$, 
let the shorthand notation $\bar{F}_n(k) := \FcfII{f(t)}{n+1}{k}$. Then the 
following expansions illustrate several characteristic forms of these 
prescribed partial sums for the first several 
special cases of \eqref{eqn_fkp_partialsum_fCf2_exp_forms} 
when $2 \leq p \leq 5$: 
\begin{align} 
\label{eqn_pth_partial_coeff_sums_p234} 
\sum_{k=1}^{n} \frac{t^{2k}}{f(k)^2} & = 
     \frac{t^{n(n+1)}}{(n!_{f})^2}\left(\bar{F}_n(2)^2 - 2 \bar{F}_n(1) \bar{F}_n(3) 
     \right) \\ 
\notag 
\sum_{k=1}^{n} \frac{t^{3k}}{f(k)^3} & = 
     \frac{t^{3n(n+1) / 2}}{(n!_{f})^3}\left(\bar{F}_n(2)^3 - 
     3 \bar{F}_n(1) \bar{F}_n(2) \bar{F}_n(3) + 3 \bar{F}_n(1)^2 \bar{F}_n(4)\right) \\ 
\notag 
\sum_{k=1}^{n} \frac{t^{4k}}{f(k)^4} & = 
     \frac{t^{4n(n+1)}}{(n!_{f})^4}\bigl(\bar{F}_n(2)^4 - 
     4 \bar{F}_n(1) \bar{F}_n(2)^2 \bar{F}_n(3) + 2 \bar{F}_n(1)^2 \bar{F}_n(3)^2 + 
     4 \bar{F}_n(1)^2 \bar{F}_n(2) \bar{F}_n(4) \\ 
\notag 
     & \phantom{= \frac{t^{4n(n+1)}}{(n!_{f})^4}\bigl( \quad \ } 
     - 4 \bar{F}_n(1)^3 \bar{F}_n(5)\bigr) \\ 
\notag 
\sum_{k=1}^{n} \frac{t^{5k}}{f(k)^5} & = 
     \frac{t^{5n(n+1) / 2}}{(n!_{f})^5}\bigl(\bar{F}_n(2)^5 - 
     5 \bar{F}_n(1) \bar{F}_n(2)^3 \bar{F}_n(3) + 5 \bar{F}_n(1)^2 \bar{F}_n(2) \bar{F}_n(3)^2 + 
     5 \bar{F}_n(1)^2 \bar{F}_n(2)^2 \bar{F}_n(4) \\ 
\notag 
   & \phantom{\frac{t^{5n(n+1) / 2}}{(n!_{f})^5}\bigl( \ \quad} - 
     5 \bar{F}_n(1)^3 \bar{F}_n(3) \bar{F}_n(4) - 
     5 \bar{F}_n(1)^3 \bar{F}_n(2) \bar{F}_n(5) + 5 \bar{F}_n(1)^4 \bar{F}_n(6)\bigr). 
\end{align} 
For each fixed integer $p > 1$, the particular partial sums defined by the 
ordinary generating function, $\widetilde{f}_n(w)$, correspond to a function in 
$n$ that is fixed with respect to the lower indices for the triangular 
coefficients defined by \eqref{eqn_genS1ft_rec_def}. 
Moreover, the resulting coefficient expansions enumerating the 
$f$-harmonic numbers at each $p \geq 2$ are isobaric in the sense that the 
sum of the indices over the lower index $k$ is $2p$ in each individual 
term in these finite sums. 
\end{example} 

\subsection{Expansions of the generalized coefficients by 
            $f$-harmonic numbers} 

The \emph{elementary symmetric polynomials} depending on the 
function $f$ implicit to the product-based definitions of the generalized 
Stirling numbers of the first kind expanded 
through \eqref{eqn_xft_gen_PHSymbol_def} 
provide new forms of the known $p$-order harmonic number, or 
\emph{exponential Bell polynomial}, expansions of the ordinary 
Stirling numbers of the first kind enumerated in the references 
\citep{STIRESUMS,COMBIDENTS,ADVCOMB,UC}. 
Thus, if we first define the weighted sums of the $f$-harmonic numbers, 
denoted $w_f(n, m)$, recursively according to an identity for the 
Bell polynomials, $\ell \cdot Y_{n,\ell}(x_1, x_2, \ldots)$, 
for $x_k \equiv (-1)^k F_n^{(k)}(t^k) (k-1)!$ as \citep[\S 4.1.8]{UC} 
\begin{align*} 
w_f(n+1, m) & := \sum_{0 \leq k < m} (-1)^{k} F_n^{(k+1)}(t^{k+1}) 
     (1-m)_k w_f(n+1,m-1-k) + \Iverson{m = 1}, 
\end{align*} 
we can expand the generalized coefficient triangles through these 
weighted sums as 
\begin{align} 
\label{eqn_FcfII_wfnm_genHNum_weighted_sum_exps_rdef} 
\FcfII{f(t)}{n+1}{k} & = \frac{n!_{f}}{(k-1)!}\ 
     w_f(n+1, k) \\ 
\notag 
     & = 
\sum_{j=0}^{k-2} \FcfII{f(t)}{n+1}{k-1-j} 
     \frac{(-1)^{j} F_n^{(j+1)}(t^{j+1})}{(k-1)} + 
     n!_{f(t)} \cdot \Iverson{k = 1}. 
\end{align} 
This definition of the weighted $f$-harmonic sums for the generalized 
triangles in \eqref{eqn_genS1ft_rec_def} implies the special case expansions 
given in the next corollary. 

\begin{cor}[Weighted $f$-Harmonic Sums for the Generalized Stirling Numbers] 
The first few special case expansions of the coefficient identities in 
\eqref{eqn_FcfII_wfnm_genHNum_weighted_sum_exps_rdef} are stated for 
fixed $f$, $t \neq 0$, and integers $n \geq 0$ in the following forms: 
\begin{align} 
\label{eqn_fCf_GenFHHarmonic_exps} 
\FcfII{f(t)}{n+1}{2} & = \frac{n!_{f}}{t^{n(n+1) / 2}}\ F_n^{(1)}(t) \\ 
\notag 
\FcfII{f(t)}{n+1}{3} & = \frac{n!_{f}}{2\ t^{n(n+1) / 2}}\left( 
     F_n^{(1)}(t)^2 - F_n^{(2)}(t^2)\right) \\ 
\notag 
\FcfII{f(t)}{n+1}{4} & = \frac{n!_{f}}{6\ t^{n(n+1) / 2}}\left( 
     F_n^{(1)}(t)^3 - 3 F_n^{(1)}(t) F_n^{(2)}(t^2) + 2 F_n^{(3)}(t^3) 
     \right) \\ 
\notag 
\FcfII{f(t)}{n+1}{5} & = \frac{n!_{f}}{24\ t^{n(n+1) / 2}}\left( 
     F_n^{(1)}(t)^4 - 6 F_n^{(1)}(t)^2 F_n^{(2)}(t^2) + 3 F_n^{(2)}(t^2)^2 + 
     8 F_n^{(1)}(t) F_n^{(3)}(t^3) - 6 F_n^{(4)}(t^4)\right). 
\end{align} 
\end{cor} 
\begin{proof} 
These expansions are computed explicitly using the recursive formula in 
\eqref{eqn_FcfII_wfnm_genHNum_weighted_sum_exps_rdef} 
for the first few cases of the lower triangle index $2 \leq k \leq 5$. 
\end{proof} 
We will return to the expansions of these coefficients 
in \eqref{eqn_FcfII_wfnm_genHNum_weighted_sum_exps_rdef} to formulate new 
finite sum identities providing functional relations between the 
$p$-order $f$-harmonic number sequences in the next section. 

\subsection{Combinatorial sums and functional equations for the 
            $f$-harmonic numbers} 

The next several properties give interesting expansions of 
the $p$-order $f$-harmonic numbers recursively over the parameter $p$ 
that can then be employed 
to remove, or at least significantly obfuscate, the 
current direct cancellation problem with these forms phrased by the 
examples in \eqref{eqn_pth_partial_coeff_sums_p234} and 
in \eqref{eqn_fCf_GenFHHarmonic_exps}. 

\begin{prop} 
For any fixed $p \geq 1$ and $n \geq 0$, we have the following 
coefficient product identities generating the 
$p$-order $f$-harmonic numbers, $F_n^{(p)}(t)$: 
\begin{align} 
\label{eqn_Fnpt_pvar_rform_exps} 
F_n^{(p+1)}(t) & = F_n^{(p)}(t) + 
     \frac{(-1)^{p} t^{n(n+1) / 2}}{t^{\frac{pn(n+1)}{2(p+1)}} n!_{f}} 
     \FcfII{f(t^{1 / (p+1)})}{n+1}{p+2} \\ 
\notag 
   & \phantom{= F_n^{(p)}(t)\ } + 
     \sum_{j=0}^{p-1} 
     \frac{p\ (-1)^{j+1} t^{n(n+1) / 2}}{t^{\frac{jn(n+1)}{2p}} 
     (n!_{f})^{p-j} (p-j)} \left( 
     \sum_{\substack{0 \leq i_1, \ldots, i_{p-j} \leq j \\ 
           i_1 + \cdots + i_{p-j} = j}} 
     \FcfII{f(t^{1 / p})}{n+1}{i_1 + 2} \cdots 
     \FcfII{f(t^{1 / p})}{n+1}{i_{p-j} + 2}\right) \\ 
\notag 
   & \phantom{= F_n^{(p)}(t)\ } + 
     \sum_{j=0}^{p-1} \sum_{i=0}^{j} 
     \frac{(p+1) t^{n(n+1) / 2} (-1)^{j}}{ 
     t^{\frac{jn(n+1)}{2(p+1)}} (n!_{f})^{p+1-j} (p+1-j)} 
     \FcfII{f(t^{1 / (p+1)})}{n+1}{i+2} \times \\ 
\notag 
   & \phantom{= F_n^{(p)}(t) + \sum\sum\ } \times 
     \left( 
     \sum_{\substack{0 \leq i_1, \ldots, i_{p-j} \leq j-i \\ 
           i_1 + \cdots + i_{p-j} = j-i}} 
     \prod_{m=1}^{p-j} \FcfII{f(t^{1 / (p+1)})}{n+1}{i_m + 2}\right). 
\end{align} 
\end{prop} 
\begin{proof} 
To begin with, observe the following rephrasing of the partial sums 
expansions from equations \eqref{eqn_fkp_partialsum_fCf2_exp_forms} and 
\eqref{eqn_fkp_partialsum_fCf2_exp_forms_v2} as 
\begin{align*} 
F_n^{(p+1)}(t) & = \frac{t^{n(n+1) / 2}}{(n!_{f})^{p+1}} \sum_{j=0}^{p} 
     \frac{(p+1)\ (-1)^j}{(p+1-j)} 
     \FcfII{f\left(t^{1 / (p+1)}\right)}{n+1}{1}^j 
     [w^{2p+2-j}] \widetilde{f}_n(w)^{p+1-j} \\ 
   & = 
     \frac{(p+1) (-1)^{p} t^{n(n+1) / 2}}{t^{\frac{pn(n+1)}{2(p+1)}} n!_{f}} 
     \FcfII{f(t^{1 / (p+1)})}{n+1}{p+2} \\ 
     & \phantom{= \quad \ } + 
     \sum_{j=0}^{p-1} 
     \frac{(p+1) (-1)^{j} t^{n(n+1) / 2}}{t^{\frac{jn(n+1)}{2(p+1)}} 
     (n!_{f})^{p+1-j} (p+1-j)} [w^{j}] \left( 
     \frac{\widetilde{f}_n(w)}{w^2}\right)^{p+1-j}. 
\end{align*} 
The coefficients involved in the partial sum forms for each sequence of 
$F_n^{(p)}(t)$ are implicitly tied to the form of 
$t \mapsto t^{1 / p}$ in the triangle definition of \eqref{eqn_genS1ft_rec_def}. 
Given this distinction, let the generating function $\widetilde{f}$ be 
defined equivalently in the more careful definition as 
$\widetilde{f}_n(w) :\equiv \widetilde{f}_n(t;\ w)$. 
The powers of the generating function $\widetilde{f}_n(w)$ from the 
previous equations satisfy the coefficient term expansions 
according to the next equation \citep[\cf \Section 7.5]{GKP}. 
\begin{align*} 
[w^{2p-j}] \widetilde{f}_n(w)^{p-j} & := 
     [w^{2p-j}] \widetilde{f}_n(t;\ w)^{p-j} = 
     [w^{j}] \left(\frac{\widetilde{f}_n(t;\ w)}{w^2}\right)^{p-j} \\ 
     & \phantom{:} = 
     \sum_{\substack{0 \leq i_1, \ldots, i_{p-j} \leq j \\ 
           i_1 + \cdots i_{p-j} = j}} 
     \FcfII{f(t)}{n+1}{i_1 + 2} \cdots \FcfII{f(t)}{n+1}{i_{p-j} + 2} 
\end{align*} 
Then by taking the 
difference of the harmonic sequence terms over successive indices 
$p \geq 1$ and at a fixed index of $n \geq 1$, the stated recurrences for 
these $p$-order sequences result. 
\end{proof} 

The generating function series over $n$ in the next proposition 
is related to the forms of the \emph{Euler sums} 
considered in \citep{STIRESUMS} and to the 
context of the generalized zeta function transformations considered in 
\citep{GFTRANSHZETA2016} briefly noted in the introduction. 
We suggest the infinite sums over these generalized 
identities for $n \geq 1$ as a 
topic for future research exploration in the concluding remarks of 
Section \ref{Section_Concl}. 

\begin{prop}[Functional Equations for the $f$-Harmonic Numbers]
\label{prop_FHNum_fnal_eqn_and_coeff_exps}
For any integers $n \geq 0$ and $p \geq 2$, we have the following 
functional relations between the $p$-order and $(p-1)$-order 
$f$-harmonic numbers over $n$ and $p$: 
\begin{align*} 
F_{n+1}^{(p)}(t^p) & = F_n^{(p)}(t^p) + 
     \sum_{1 \leq j < p} \gkpSI{n+2}{p+1-j}_{f(t)} 
     \frac{(-1)^{p+1-j} t^{j(n+1)}}{f(n+1)^{j} (n+1)!_{f(t)}} + 
     \gkpSI{n+1}{p}_{f(t)} \frac{(-1)^{p+1}}{(n+1)!_{f(t)}} \\ 
     & = F_n^{(p)}(t^p) + \frac{t^{(p-1)(n+1)}}{f(n+1)^{p-1}} + 
     \frac{(-1)^{p-1}}{(n+1)!_{f(t)}} \left( 
     \gkpSI{n+1}{p}_{f(t)} + \gkpSI{n+1}{p-1}_{f(t)} \right) \\ 
     & \phantom{= F_n^{(p)}(t^p) \ } + 
     \gkpSI{n+2}{p}_{f(t)} \frac{(-1)^{p} t^{n+1}}{f(n+1) (n+1)!_{f(t)}} \\ 
     & \phantom{= F_n^{(p)}(t^p) \ } + 
     \sum_{j=0}^{p-3} \gkpSI{n+2}{j+2}_{f(t)} 
     \frac{(-1)^{j+1} \left(f(n+1)t^{-(n+1)} - 1\right) t^{(p-1-j)(n+1)}}{ 
     f(n+1)^{p-1-j} (n+1)!_{f(t)}}. 
\end{align*} 
\end{prop} 
\begin{proof} 
First, notice that \eqref{eqn_FcfII_wfnm_genHNum_weighted_sum_exps_rdef} 
implies that we have the following weighted harmonic number sums for the 
$p$-order $f$-harmonic numbers: 
\begin{align*}
F_n^{(p)}(t^p) & = \sum_{1 \leq j < p} \gkpSI{n+1}{p+1-j}_{f(t)} 
     \frac{(-1)^{p+1-j} F_n^{(j)}(t^j)}{n!_{f(t)}} + 
     \gkpSI{n+1}{p+1}_{f(t)} \frac{p (-1)^{p+1}}{n!_{f(t)}}. 
\end{align*} 
Next, we use \eqref{eqn_genS1ft_rec_def} twice to expand the 
differences of the left-hand-side of the previous equation as 
\begin{align*} 
\frac{t^{p(n+1)}}{f(n+1)^p} & = F_{n+1}^{(p)}(t^p) - F_n^{(p)}(t^p) \\ 
     & = 
     \sum_{1 \leq j < p} \gkpSI{n+2}{p+1-j}_{f(t)} 
     \frac{(-1)^{p+1-j} F_{n+1}^{(j)}(t^j)}{(n+1)!_{f(t)}} - 
     \sum_{1 \leq j < p} \gkpSI{n+1}{p+1-j}_{f(t)} 
     \frac{(-1)^{p+1-j} F_{n}^{(j)}(t^j)}{n!_{f(t)}} \\ 
     & \phantom{= \sum \quad \ } + 
     \gkpSI{n+2}{p+1}_{f(t)} \frac{p (-1)^{p+1}}{(n+1)!_{f(t)}} - 
     \frac{f(n+1)}{t^{n+1}} \gkpSI{n+1}{p+1}_{f(t)} 
     \frac{p (-1)^{p+1}}{(n+1)!_{f(t)}} \\ 
     & = 
     \sum_{1 \leq j < p} \gkpSI{n+2}{p+1-j}_{f(t)} 
     \frac{(-1)^{p+1-j} t^{j(n+1)}}{f(n+1)^{j} (n+1)!_{f(t)}} - 
     \sum_{1 \leq j < p} \gkpSI{n+1}{p-j}_{f(t)} 
     \frac{(-1)^{p-j} F_n^{(j)}(t^j)}{(n+1)!_{f(t)}} \\ 
     & \phantom{= \sum \quad \ } + 
     \gkpSI{n+1}{p}_{f(t)} \frac{p (-1)^{p+1}}{(n+1)!_{f(t)}} \\ 
     & = 
     \sum_{1 \leq j < p} \gkpSI{n+2}{p+1-j}_{f(t)} 
     \frac{(-1)^{p+1-j} t^{j(n+1)}}{f(n+1)^{j} (n+1)!_{f(t)}} - 
     \gkpSI{n+1}{p}_{f(t)} \frac{(p-1) (-1)^{p+1}}{(n+1)!_{f(t)}} \\ 
     & \phantom{= \sum \quad \ } + 
     \gkpSI{n+1}{p}_{f(t)} \frac{p (-1)^{p+1}}{(n+1)!_{f(t)}}. 
\end{align*} 
The second identity is verified similarly by combining the coefficient 
terms as in the last equations and adding the right-hand-side 
differences of the $(p-1)$-order $f$-harmonic numbers to the first identity. 
\end{proof} 

One immediate corollary that must by its importance be expanded in turn 
explicitly in the next example 
provides new expansions of the $p$-order harmonic numbers in terms of the 
ordinary triangle of Stirling numbers of the first kind corresponding to the 
case where $(f(n), t) \equiv (n, 1)$ in the previous proposition. 
Similar expansions of identities related to the generalized 
generating function transformations in \citep{GFTRANSHZETA2016} result for the 
special cases of the proposition where 
$(f(n), t) \equiv (\alpha n+\beta, t)$ for some application-dependent 
prescribed $\alpha, \beta \in \mathbb{C}$ 
defined such that $-\frac{\beta}{\alpha} \notin \mathbb{Z}$. 
Another special case worth noting and independently expanding 
provides analogous relations between the 
$q$-binomial coefficients implicit to the forms of the 
\emph{$q$-binomial theorem} expanding the $q$-Pochhammer symbols, 
$(a; q)_n$, for each $n \geq 0$ \citep[\cf \S 17.2]{NISTHB}. 

\begin{example}[Stirling Numbers and Euler Sums] 
\label{example_SpCase_S1HNum_FnalEqn_Ident}
For all integers $p \geq 3$ and fixed $n \in \mathbb{Z}^{+}$, 
we have the following identity relating the successive differences of the 
$p$-order harmonic numbers and the Stirling numbers of the first kind: 
\begin{align} 
\label{eqn_S1HNum_fnaleqn_exp_v1} 
\frac{1}{n^p} & = \frac{1}{n^{p-1}} + \frac{(-1)^{p-1}}{n!}\left( 
     \gkpSI{n}{p} + \gkpSI{n}{p-1}\right) + 
     \gkpSI{n+1}{p} \frac{(-1)^p}{n \cdot n!} \\ 
\notag 
     & \phantom{=\frac{1}{n^{p-1}}\ } + 
     \sum_{j=0}^{p-3} \gkpSI{n+1}{j+2} \frac{(-1)^{j+1} (n-1)}{ 
     n^{p-1-j} \cdot n!}. 
\end{align} 
The relation in \eqref{eqn_S1HNum_fnaleqn_exp_v1} certainly implies new 
finite sum identities between the $p$-order harmonic numbers and the 
Stirling numbers of the first kind, though the generating functions and 
limiting cases of these sums provide more information on infinite sums 
considered in several of the references. 

With this in mind, we define the \emph{Nielsen generalized polylogarithm}, 
$S_{t,k}(z)$, by the infinite generating series over the $t$-power-scaled 
Stirling numbers as \citep[\cf \S 5]{STIRESUMS} 
\begin{align*} 
S_{t,k}(z) & := \sum_{n \geq 1} \gkpSI{n}{k} \frac{z^n}{n^t \cdot n!}. 
\end{align*} 
We see immediately that \eqref{eqn_S1HNum_fnaleqn_exp_v1} provides strictly 
enumerative relations between the polylogarithm function 
generating functions, $\operatorname{Li}_p(z) / (1-z)$, for the 
$p$-order harmonic numbers and the Nielsen polylogarithms. 
Perhaps more interestingly, we also find new identities between the 
Riemann zeta functions, 
$\zeta(p)$ and $\zeta(p-1)$, and the special classes of \emph{Euler sums} 
given by $S_{t,k}(1)$ for $t \in [2, p-1]$ and $k \in [2, p]$ 
defined as in the reference \citep[\S 5]{STIRESUMS}. 
\end{example} 

\section{Coefficient identities and 
         generalized forms of the Stirling convolution polynomials} 

\subsection{Generalized Coefficient Identities and Relations} 

There are several immediate for small-indexed columns of the triangle 
defined by \eqref{eqn_genS1ft_rec_def} and 
that can both be given immediately and that follow from an 
inductive argument. The 
next identities in \eqref{eqn_gen_FcfII_ftnk_gen_k_idents} are given 
for general lower column index $k \geq 1$ by  
\begin{align} 
\label{eqn_gen_FcfII_ftnk_gen_k_idents} 
\FcfII{f(t)}{n}{k} & = 
     [w^{k-1}]\left(\prod_{j=1}^{n-1} (w + f(j)\ t^{-j})\right) 
     \Iverson{n \geq 1} + \Iverson{n = k = 0} \\ 
\notag 
     & = 
     \sum_{\substack{ 
           0 < i_1 < \cdots < i_{n-k} < n}} 
     f(i_1) \cdots f(i_{n-k}) \cdot t^{-(i_1 + \cdots + i_{n-k})}, 
\end{align} 
which follows immediately by considering the first products of the form 
$\prod_i (z + x_i)$ in the context of elementary symmetric polynomials for 
these specific $x_i$. 

\begin{prop}[Horizontal and Vertical Column Recurrences] 
The generalized Stirling numbers of the first kind over the 
first several special case columns for the 
shifted upper index of $n+1$ in the expansions of 
\eqref{eqn_genS1ft_rec_def} are given by the 
next recurrence relations for all $n \geq 0$ and any $k \geq 2$. 
\begin{align} 
\label{eqn_FcfII_ftnk_spcase_cols_and_rdefs} 
\FcfII{f(t)}{n+1}{1} & = \frac{n!_{f}}{t^{n(n+1) / 2}} \\ 
\notag 
\FcfII{f(t)}{n+1}{k} & = \frac{n!_{f}}{t^{n(n+1) / 2}} \sum_{j=1}^{n} 
     \FcfII{f(t)}{j}{k-1} \frac{t^{j(j+1) / 2}}{j!_{f}},\ 
     \text{ if $k \geq 2$} 
\end{align} 
\end{prop} 
\begin{proof} 
We begin by observing that by \eqref{eqn_genS1ft_rec_def} 
when $k \equiv 1$, we have that 
\begin{align*} 
\gkpSI{n+1}{1}_{f(t)} & = \frac{f(n)}{t^n} \gkpSI{n}{1}_{f(t)} + 
     \gkpSI{n}{0}_{f(t)} \\ 
     & = 
     \frac{f(n)}{t^n} \gkpSI{n}{1}_{f(t)} + \Iverson{n = 0}, 
\end{align*} 
which implies the first claim by induction since 
$\gkpSI{1}{1}_{f(t)} = 1$ and $\gkpSI{0}{1}_{f(t)} = 1$. 
To prove the column-wise recurrence relation given in 
\eqref{eqn_FcfII_ftnk_spcase_cols_and_rdefs}, 
we notice again by induction that for any functions 
$g(n)$ and $b(n) \neq 0$, the sequence, $f_k(n)$, defined recursively by 
\begin{align*} 
f_k(n) & = 
     \begin{cases} 
        b(n) \cdot f_k(n-1) + g(n-1) & \text{ if $n \geq 1$ } \\ 
        1 & \text{ if $n = 0$, } 
     \end{cases} 
\end{align*} 
has a closed-form solution given by 
\begin{align*} 
f_k(n) & = \left(\prod_{j=1}^{n-1} b(j)\right) \times \sum_{0 \leq j < n} 
     \frac{g(j)}{\prod_{i=1}^{j} b(j)}. 
\end{align*} 
Thus by \eqref{eqn_genS1ft_rec_def} the second claim is true. 
\end{proof} 

\subsection{Generalized forms of the Stirling convolution polynomials} 

\begin{definition}[Stirling Polynomial Analogs] 
\label{def_CvlPolyAnalogs}
For $x,n,x-n \geq 1$, we suggest the next two variants of the 
generalized \emph{Stirling convolution polynomials}, denoted by 
$\sigma_{f(t),n}(x)$ and $\widetilde{\sigma}_{f(t),n}(x)$, respectively, 
as the right-hand-side coefficient definitions in the following equations: 
\begin{align} 
\label{eqn_fnx_poly_coeff_def} 
\sigma_{f(t),n}(x) := \FcfII{f(t)}{x}{x-n} \frac{(x-n-1)!}{x!_{f}} 
     & \quad \iff \quad 
\FcfII{f(t)}{n+1}{k} = \frac{(n+1)!_{f}}{(k-1)!}\ \sigma_{f(t),n+1-k}(n+1) \\ 
\notag 
\widetilde{\sigma}_{f(t),n}(x) := \FcfII{f(t)}{x}{x-n} \frac{(x-n-1)!}{x!} 
     & \quad \iff \quad 
\FcfII{f(t)}{n+1}{k} = \frac{(n+1)!}{(k-1)!}\ \widetilde{\sigma}_{f(t),n+1-k}(n+1). 
\end{align} 
\end{definition} 

\begin{prop}[Recurrence Relations] 
For integers $x,n,x-n \geq 1$, the analogs to the Stirling convolution 
polynomial sequences defined by \eqref{eqn_fnx_poly_coeff_def} each 
satisfy a respective recurrence relation stated in the next equations. 
\begin{align} 
\notag 
f(x+1) \sigma_{f(t),n}(x+1) & = (x-n) \sigma_{f(t),n}(x) + 
     f(x)\ t^{-x} \cdot \sigma_{f(t),n-1}(x) + \Iverson{n = 0} \\ 
\label{eqn_fnx_snx_genCvlPolySeqs_recs} 
(x+1) \widetilde{\sigma}_{f(t),n}(x+1) & = (x-n) \widetilde{\sigma}_{f(t),n}(x) + 
     f(x)\ t^{-x} \cdot \widetilde{\sigma}_{f(t),n-1}(x) + \Iverson{n = 0} 
\end{align} 
\end{prop} 
\begin{proof} 
We give a proof of the second identity since the first recurrence follows 
almost immediately from this result. 
Let $x,n,x-n \geq 1$ and consider the expansion of the left-hand-side of 
\eqref{eqn_fnx_snx_genCvlPolySeqs_recs} according to 
Definition \ref{def_CvlPolyAnalogs} as follows: 
\begin{align*} 
(x + 1) \widetilde{\sigma}_{f(t),n}(x + 1) & = 
     \gkpSI{x+1}{x+1-n}_{f(t)} \frac{(x-n)!}{x!} \\ 
     & = 
     \left(f(x) t^{-x} \gkpSI{x}{x+1-n}_{f(t)} + \gkpSI{x}{x-n}_{f(t)} 
     \right) (x-n) \cdot \frac{(x-n-1)!}{x!} \\ 
     & = 
     (x-n) \widetilde{\sigma}_{f(t),n}(x) + 
     f(x) t^{-x} \cdot \widetilde{\sigma}_{f(t),n-1}(x). 
\end{align*} 
For any non-negative integer $x$, when $n = 0$, we see that 
$\gkpSI{x+1}{x+1}_{f(t)} \equiv 1$, which implies the result. 
\end{proof} 

\begin{remark}[A Comparison of Polynomial Generating Functions] 
The generating functions for the Stirling convolution polynomials, 
$\sigma_n(x)$, and the $\alpha$-factorial polynomials, $\sigma_n^{(\alpha)}(x)$, 
from \citep{MULTIFACTJIS} each have the comparatively simple 
special case closed-form generating functions given by 
\begin{align} 
\label{eqn_SPoly_def_and_GF}
x \sigma_n(x) & = \gkpSI{x}{x-n} \frac{(x-n-1)!}{(x-1)!} = 
     [z^n] \left(\frac{z e^{z}}{e^{z}-1}\right)^{x} && 
     \text{ for } (f(n), t) \equiv (n, 1) \\ 
\notag 
x \sigma_n^{(\alpha)}(x) & = \gkpSI{x}{x-n}_{\alpha} \frac{(x-n-1)!}{(x-1)!} = 
     [z^n] e^{(1-\alpha)z} 
     \left(\frac{\alpha z e^{\alpha z}}{e^{\alpha z}-1}\right)^{x} && 
     \text{ for } (f(n), t) \equiv (\alpha n + 1 - \alpha, 1) \\ 
\notag 
x \sigma_n^{(\alpha; \beta)}(x) & = \gkpSI{x}{x-n}_{(\alpha; \beta)} 
     \frac{(x-n-1)!}{(x-1)!} = 
     [z^n] e^{\beta z} 
     \left(\frac{\alpha z e^{\alpha z}}{e^{\alpha z}-1}\right)^{x} && 
     \text{ for } (f(n), t) \equiv (\alpha n + \beta, 1). 
\end{align} 
The Stirling polynomial sequence in \eqref{eqn_SPoly_def_and_GF} is a 
special case of a more general class of \emph{convolution polynomial} 
sequences defined by Knuth in his article \citep{CVLPOLYS}. 

These polynomial sequences are defined by a general sequence of 
coefficients, $s_n^{\ast}$ with $s_0^{\ast} = 1$, 
such that the corresponding polynomials, $s_n(x)$, are 
enumerated by the power series over the original sequence as 
\begin{equation*} 
\sum_{n=0}^{\infty} s_n(x) z^n := S(z)^{x} \equiv 
     \left(1 + \sum_{n=1}^{\infty} s_n^{\ast} z^n\right)^{x}. 
\end{equation*} 
Polynomial sequences of this form satisfy a number of interesting 
properties, and in particular, the next identity provides a 
generating function for a variant of the 
original convolution polynomial sequence over $n$ when 
$t \in \mathbb{C}$ is fixed. 
\begin{equation} 
\label{eqn_CvlPoly_Stz_GF_rdef} 
\mathcal{S}_t(z) := S\left(z \mathcal{S}_t(z)^t\right) 
     \quad \implies \quad 
\frac{x s_n(x+tn)}{(x+tn)} = [z^n] \mathcal{S}_t(z)^{x} 
\end{equation} 
This result is also useful in expanding many identities 
for the $t := 1$ case as given for the Stirling polynomial case in 
\citep[\Section 6.2]{GKP} \citep{CVLPOLYS}. 
A related generalized class of polynomial sequences is considered in 
Roman's book defining the form of 
\emph{Sheffer polynomial} sequences. \nocite{UC} 
The polynomial sequences of this particular type, say with sequence 
terms given by $s_n(x)$, satisfy the form in the following 
generating function identity where $A(z)$ and $B(z)$ are 
prescribed power series satisfying the initial conditions from the 
reference \citep[\cf \Section 2.3]{UC}: 
\begin{equation*} 
\sum_{n=0}^{\infty} s_n(x) \frac{z^n}{n!} := A(z) e^{x B(z)}. 
\end{equation*} 
For example, the form of the generalized, or higher-order Bernoulli 
polynomials (numbers) is a parameterized sequence whose generating 
function yields the form of many other special case sequences, 
including the Stirling polynomial case defined in equation 
\eqref{eqn_SPoly_def_and_GF} \citep[\cf \Section 4.2.2]{UC} 
\citep[\cf \Section 5]{MULTIFACTJIS}. 
\end{remark} 

\subsubsection*{An experimental procedure towards evaluating the 
                generalized polynomials} 

We expect that the generalized convolution polynomial analogs defined in 
\eqref{eqn_fnx_poly_coeff_def} above form a 
sequence of finite-degree polynomials in $x$, for example, as in the 
Stirling polynomial case when we have that 
\begin{align*} 
\gkpSI{x}{x-n} & = \sum_{k \geq 0} \gkpEII{n}{k} \binom{x+k}{2n}, 
\end{align*} 
where $\gkpEII{n}{k}$ denotes the special triangle of 
\emph{second-order Eulerian numbers} for $n, k \geq 0$ and where the 
binomial coefficient terms in the previous equations each have a finite-degree
polynomial expansion in $x$ \citep[\S 6.2]{GKP}. 
The previous identity also allows us to extend the 
Stirling numbers of the first kind to \emph{arbitrary} real, 
or complex-valued inputs. 

Given the 
relatively simple and elegant forms of the generating functions that 
enumerate the polynomial sequences of the special case forms in 
\eqref{eqn_SPoly_def_and_GF}, it seems natural to 
attempt to extend these relations to the generalized polynomial 
sequence forms defined by \eqref{eqn_fnx_poly_coeff_def}. 
However, in this more general context we appear to have a 
stronger dependence of the form and ordinary generating functions of these 
polynomial sequences on the underlying function $f$. 
Specifically, for the form of the first sequence in 
\eqref{eqn_fnx_poly_coeff_def}, 
we suppose that the function $f(n)$ is arbitrary. 

Based on the first several cases of these polynomials, 
it appears that the generating function for the sequence can be 
expanded as 
\begin{align} 
\label{eqn_fnx_poly_GF_ident_v1} 
 & f_n(x) := [z^n] F(z)^{x} \quad \text{ where } \quad 
     F(z) := \sum_{n=0}^{\infty} g_n(x) z^n \\ 
\notag 
 & \phantom{f_n(x) :} \implies  
     g_n(x) = \frac{\sum_{j=0}^{n-1} f(x)^n \numpoly_n(j;\ x) x^{n-1-j} 
     (1+x)^j f(x+1)^j}{n!\ t^{nx}\ \sum_{j=0}^{2n-1} 
     \denompoly_n(j;\ x) x^{2n-1-j} 
     (1+x)^j f(x+1)^j}\ \Iverson{n \geq 1} + \Iverson{n = 0} 
\end{align} 
where the forms $\numpoly_n(j;\ x)$ and $\denompoly_n(j;\ x)$ denote 
polynomial sequences of finite non--negative integral degree indexed over the 
natural numbers $n, j \geq 0$. 
Similarly it has been verified for the first $16$ of each $n$ and $k$ 
that the following equation holds where the 
terms $g_n(x)$ involved in the series for $F(z)$ are defined through the 
form of the last equation. 
\begin{equation*} 
s_n(k) := f_{n-k}(n) \implies s_n(k) = [z^n] z^k F(z)^n = 
     \sum_{j=1}^{n-k} \binom{n}{j} [z^{n-k}] (F(z) - 1)^j + \Iverson{n = k} 
\end{equation*} 
Note that the coefficients defined through these implicit power series forms 
must also satisfy an implicit relation to the particular values of the 
polynomial parameter $x$ as formed through the last equations, which is 
much different in construction than in the cases of the 
special polynomial sequence generating functions remarked on above. 
Other different expansions may result for special cases of the 
function $f(n)$ and explicit values of the parameter $t$. 

\section{Conclusions and future research} 
\label{Section_Concl}

\subsection{Summary} 

We have defined a generalized class of factorial product functions, 
$(x)_{f(t),n}$, that generalizes the forms of many special and symbolic 
factorial functions considered in the references. 
The coefficient-wise symbolic polynomial expansions of these 
$f$-factorial function variants define generalized triangles of 
Stirling numbers of the first kind which share many analogs to the 
combinatorial properties satisfied by the ordinary combinatorial 
triangle cases. 
Surprisingly, many inversion relations and other finite sum properties 
relating the ordinary Stirling number triangles are not apparent by 
inspection of these corresponding sums in the most general cases. 
A study of ordinary Stirling-number-like sums, inversion relations, and 
generating function transformations is not contained in the article. 
We pose formulating these analogs in the most general coefficient cases as a 
topic for future combinatorial work with the generalized Stirling number 
triangles defined in 
Section \ref{subSection_Intro_GenSNumsDefs}. 

\subsection{Topics suggested for future research} 

Another new avenue to explore with these sums and the generalized $f$-zeta 
series transformations motivated in \citep{GFTRANS2016,GFTRANSHZETA2016} is 
to consider finding new identities and expressions for the 
Euler-like sums suggested by the generalized identity in 
Proposition \ref{prop_FHNum_fnal_eqn_and_coeff_exps} and by the 
special case expansions for the Stirling numbers of the first kind given in 
Example \ref{example_SpCase_S1HNum_FnalEqn_Ident}. 
In particular, 
if we define a class of so-termed ``\emph{$f$-zeta}'' functions, 
$\zeta_f(s) := \sum_{n \geq 1} f(n)^{-s}$, we seek analogs to these infinite 
Euler sum variants expanded through $\zeta_f(s)$ just as the Euler sums are 
expressed through sums and products of the \emph{Riemann zeta function}, 
$\zeta(s)$, in the ordinary cases from \citep{STIRESUMS}. 

For example, it is well known that for real-valued $r > 1$ 
\begin{align*} 
\sum_{n \geq 1} \frac{H_n^{(r)}}{n^r} & = 
     \frac{1}{2}\left(\zeta(r)^2 + \zeta(2r)\right), 
\end{align*} 
and moreover, summation by parts shows us that for any real 
$r > 1$ and any $t \in \mathbb{C}^{\ast}$ such that we have a 
convergent limiting zeta function series we have that 
\begin{align*} 
\sum_{n \geq 1} \frac{F_n^{(r)}(t^r) t^{rn}}{f(n)^{r}} & = 
     \lim_{n\longrightarrow\infty}\ \left\{ 
     \left(F_n^{(r)}(t^r)\right)^2 - \sum_{0 \leq j < n} 
     \frac{F_j^{(r)}(t^r) t^{r(j+1)}}{f(j+1)^r} 
     \right\} \\ 
     & = \lim_{n\longrightarrow\infty}\ \left\{ 
     \left(F_n^{(r)}(t^r)\right)^2 - \sum_{0 \leq j < n} 
     \frac{F_{j+1}^{(r)}(t^r) t^{r(j+1)}}{f(j+1)^r} + 
     \sum_{0 \leq j < n} \frac{t^{2r(j+1)}}{f(j+1)^{2r}} 
     \right\}, 
\end{align*} 
which similarly implies that 
\begin{align*} 
\sum_{n \geq 1} \frac{F_n^{(r)}(1)}{f(n)^r} & 
     \quad \overset{: \rightsquigarrow}{\longrightarrow} \quad 
     \frac{1}{2}\left(\zeta_f(r)^2 + \zeta_f(2r)\right). 
\end{align*} 
Additionally, we seek other analogs to known identities for the infinite 
Euler-like-sum variants over the weighted $f$-harmonic number sums of the form 
\begin{align*} 
H_f\left(\varpi_1, \ldots, \varpi_k; s, t, z\right) & := \sum_{n \geq 1} 
     \frac{F_n^{\left(\varpi_1\right)}\left(t^{\varpi_1}\right) \cdots 
     F_n^{\left(\varpi_k\right)}\left(t^{\varpi_k}\right) z^{sn}}{f(n)^{s}}, 
\end{align*} 
when $t = \pm 1$, or more generally for any fixed $t \in \mathbb{C}^{\ast}$, 
and where the right-hand-side series in the previous equation 
converges, say for $|z| \leq 1$. 

\renewcommand{\refname}{References}

\bigskip\hrule\bigskip 

\end{document}